\documentclass[letterpaper, 10 pt, conference]{IEEEconf3}  

\IEEEoverridecommandlockouts                              

\usepackage{amsmath,amssymb,amsfonts}
\usepackage{cite}
\usepackage[dvipsnames]{xcolor}
\usepackage{algorithmic}
\usepackage{graphicx}
\usepackage{textcomp}
\usepackage{xcolor}
\usepackage{dsfont}
\usepackage{float}
\usepackage{tikz}
\usetikzlibrary{positioning}
\usepackage{calc}
\usepackage{comment}
\usepackage[normalem]{ulem}
\usepackage{resizegather}
\usepackage{etoolbox} 

\newcommand{\edit}[1]{#1}

\newcommand{\R}{{\mathbb{R}}}

\newcommand{\C}{{\mathbb{C}}}
\newcommand{\N}{{\mathbb{N}}}
\newcommand{\Z}{{\mathbb Z}}

\DeclareMathOperator{\sech}{sech}

\usepackage{amsthm}

\newtheoremstyle{ieeeconf}
  {0pt}   
  {0pt}   
  {\normalfont}  
  {\parindent}       
  {\itshape} 
  {:}         
  { } 
  {\thmname{#1} \thmnumber{#2}\thmnote{ (#3)}} 
\makeatletter
\renewenvironment{proof}[1][\proofname]{\par
  \pushQED{\qed}%
  \normalfont \topsep\z@
  \trivlist
  \item[\hskip2em
        \itshape
    #1\@addpunct{:}]\ignorespaces
}{%
  \popQED\endtrivlist\@endpefalse
}
\makeatletter

\theoremstyle{ieeeconf}

\newtheorem{theorem}{Theorem}

\newtheorem{lemma}[theorem]{Lemma}

\newtheorem{assumption}{Assumption}
\newtheorem{definition}{Definition}

\title{\LARGE \bf
Spatially-invariant opinion dynamics on the circle 
} 

\author{Giovanna Amorim$^{1}$, Anastasia Bizyaeva$^{2}$, Alessio Franci$^{3\dagger}$, and Naomi Ehrich Leonard$^{1\dagger}$
\thanks{This work was supported in part by ONR grant N00014-19-1-2556.}
\thanks{$^{1}$ Mechanical and Aerospace Engineering, Princeton University, Princeton, NJ 08544 USA, \{{\tt\small giamorim, naomi}\}{\tt\small @princeton.edu}}%
\thanks{$^{2}$  Sibley School of Mechanical and Aerospace Engineering, Cornell University, Ithaca, NY 14850 USA, {\tt \small anastasiab@cornell.edu}}
\thanks{$^{3}$ Electrical Engineering and Computer Science, University of Liege, and  WEL Research Institute., Wavre, Belgium, {\tt\small afranci@uliege.be}}%
\thanks{$^{\dagger}$These authors contributed equally to this work.}
}

\begin{document}

\maketitle
\thispagestyle{empty}
\pagestyle{empty}
\begin{abstract}
We propose \edit{and analyze} a nonlinear opinion dynamics model for an agent making decisions about a continuous distribution of options in the presence of input. Inspired by perceptual decision-making, we develop \edit{new} theory for \edit{opinion formation in response to inputs about} options distributed on the circle. \edit{Options on the circle can} represent, e.g., the possible \edit{directions of perceived objects and resulting} heading directions in planar robotic navigation problems. Interactions among options are encoded through a spatially invariant kernel, \edit{which we design to ensure that 
only a 
small (finite)
subset of options can be favored over the continuum. 
We leverage the spatial invariance of the model linearization to design flexible, distributed opinion-forming behaviors using spatiotemporal frequency domain and bifurcation analysis. }
We illustrate our model's versatility with \edit{an application to} robotic navigation
in crowded spaces. 
\end{abstract}

\section{Introduction}\label{sec:intro}

In perceptual decision-making, animals use sensory information, such as visual and auditory \edit{stimuli}, to respond to their environment. 
\edit{
Spatial invariance, the ability to respond to stimuli based solely on relative positions rather than absolute spatial coordinates, is believed to be a key feature of these sensory processes~\cite{duhamel1997spatial}.
Inspired by these insights, 
neural field models of perceptual decision-making
leverage spatial invariance~\cite{amari_dynamics_1977,ferreira2016multibump,ermentrout_spatiotemporal_2014, bressloff2012spatiotemporal, bressloff2002geometric}.}
They describe the spatiotemporal dynamics of neural activity using integro-differential equations with a
convolution kernel that captures 
interactions between different regions of the neural field.


\edit{These models are widely used for embodied intelligence, where sensory input, actions, and cognitive processes are interconnected.
 In robotics, an agent can use a distributed representation of its visual field and the objects within it to drive decisions. For example, neural field models have been used for robotic navigation in unknown environments with obstacles~\cite{oubbati2006,dahm1998robot},  manipulation~\cite{dahm1998robot}, target acquisition~\cite{schoner1995},  sensorimotor control of robots through coupled fields~\cite{erlhagen2006}, and modeling of cognitive intentions~\cite{tekulve2024intagent}. 
 Neural fields are also used in neuromorphic devices, which emulate biological processing in extremely low power hardware~\cite{indiveri2019neuromoprhiag}.

While these applications highlight the versatility of neural fields for embodied intelligence, they mostly rely on empirical approaches.} There are analytical approaches that characterize the behavior of neural field models~\cite{bressloff2012spatiotemporal, amari_dynamics_1977, bressloff2002geometric, ermentrout_spatiotemporal_2014, ferreira2016multibump}, but their input-output behavior for arbitrary inputs is not yet fully characterized.
Response to input is considered in \cite{amari_dynamics_1977,ferreira2016multibump,ermentrout_spatiotemporal_2014} but only for specific classes of inputs. \edit{Our work here lays a theoretical foundation for analysis, design, and control in more general scenarios. The novelty of our contribution lies in our study of the input-output behavior of our proposed nonlinear neural field model.}
We use a spatiotemporal transfer function to predict the model's response from its linearization.



\edit{We propose a neural field model to generalize nonlinear opinion dynamics (NOD)~\cite{LeoBizFra_AnnualReview2024} from a finite set to a continuum of options.} 
NOD has been used for robotic perceptual decision-making in obstacle avoidance and task allocation scenarios~\cite{cathcart2023proactive,amorim2024threshold}.
\edit{The distributed NOD model does not require}
prior knowledge of the number of objects in an agent's visual field and captures object volume and distance in its continuous representation.



Our contributions are as follows. First, we propose a \edit{new} nonlinear opinion dynamics model for an agent making decisions about a continuous distribution of options \edit{on the circle} \edit{\textit{and}} in the presence of input.
Second, we prove the \edit{system-theoretic} spatial invariance of the model linearization. 
Third, we use spatial-invariance of the linearized dynamics to prove the existence of an opinion forming bifurcation for the model with zero input. 
Fourth, we use space and time frequency domain analysis of the model linearization \edit{and define a spatiotempotal transfer function} to infer the input-output behavior of the nonlinear dynamics to arbitrary inputs.
Fifth, we propose a framework for designing kernels for an application of our model to robotic 
navigation.

Mathematical background is in Section \ref{sec:prelims}. We present the model in Section \ref{sec:model}. We prove the spatial invariance of the model linearization in Section \ref{sec:sys_analysis}. 
In Section \ref{sec:bif_theory} we prove an opinion-forming bifurcation in the model with zero input. In Section \ref{sec:input_response}, we discuss the model's input-output behavior. We propose a kernel design approach and illustrate our approach in Section \ref{sec:application}. A discussion is provided in Section \ref{sec:discussion}.





\section{Mathematical Preliminaries}
\label{sec:prelims}

We denote the set of integer values as $\Z$,
the set of non-negative integer values as $\N$, 
the set of real numbers as $\R$, and the set of complex numbers as $\C$. The unit circle is denoted by $\mathbb{S}^1$, i.e., $\mathbb{S}^1 = \R/\Z$. 
For $a,b\in \R$, the notation $a \nearrow b$ indicates the limit $a\to b$ 
with $a<b$.
For a complex number $s = \sigma + i \omega$, the real and imaginary parts are denoted as $\Re(s)$ and $\Im(s)$, respectively. We represent the complex conjugate as $\bar{s} = \sigma - i\omega$, the modulus as $|s| = \sqrt{s\bar{s}}$ and the argument as $\arg(s) = \lim_{n\to \infty}n\Im(\sqrt[n]{s/|s|})$ for $n\in\N - \{0\}$. 

The Hilbert space of square-integrable real functions on $\mathbb{S}^1$ is denoted by $L_2(\mathbb{S}^1)$.
\edit{The inner product} of
$v,w\in L_2(\mathbb{S}^1)$ is $\langle v, w \rangle = \int_{\mathbb{S}^1}v(\theta)w(\theta)d\theta$. 
The induced norm is $||v|| = \langle v, v\rangle^{1/2}$. 
We denote operators with capital 
letters. 
Let $A:L_2(\mathbb{S}^1) \to L_2(\mathbb{S}^1)$ be a linear operator.
We let the set 
$\text{Sp}(A) = \{\lambda_k\}$ denote the point spectrum of $A$, if it is not empty. Each eigenvalue $\lambda_{k} \in \text{Sp}(A)$ satisfies $Av_k(\theta) = \lambda_kv_k(\theta)$, where $v_k\in L_2(\mathbb{S}^1)$ denotes the eigenfunction corresponding to $\lambda_{k}$. 
We denote $\lambda_{\max} = \arg\!\max\{\Re(\lambda_k)\}$ as the \textit{leading eigenvalue} of $A$, and its corresponding eigenfunction, $v_{\max}\in L_2(\mathbb{S}^1)$, as the \textit{leading mode}.

\begin{definition}[Differential operator] Let $F:L_2(\mathbb{S}^1) \to L_2(\mathbb{S}^1)$ be a nonlinear operator. The differential of $F$ in the direction of $z$ at a point $z^*$, 
is \edit{$A_F = D_{z}F(z^*):=\lim_{\epsilon \to 0} \frac{1}{\epsilon}\big(F(\epsilon z + z^*) - F(z^*)\big) ,$}
 provided that the limit exists.
\end{definition}

\begin{definition}[Multiplication Operator] 
 A \emph{multiplication operator} $M$ is defined by $[Mh](x) := M(x)h(x)$, where $h$ is in the domain of $M$. Multiplication operators are the infinite-dimensional equivalent of diagonal matrices.
 \label{def:mult_op}
\end{definition}


 \begin{definition}[Spatial shift operator \cite{bamieh2002}, \cite{arbelaiz2024optimalestimationspatiallydistributed}] The spatial shift operator denoted by $T_\psi:L^2(\mathbb{S}^1) \to L^2(\mathbb{S}^1)$ is defined as $h(\theta) \mapsto [T_\psi h](\theta):=h(\theta-\psi)$
    for $\psi \in \mathbb{S}^1$ and $h \in  L^2(\mathbb{S}^1)$.
    \end{definition}
    
\begin{definition}[Spatially invariant operator \cite{bamieh2002}, \cite{arbelaiz2024optimalestimationspatiallydistributed}]

An operator $F$ is spatially invariant if $T_\psi F = FT_\psi$. 
\label{def:spatial_invt_op}
\end{definition}
We mainly work with a special class of spatially invariant linear operators, namely, \textit{spatial convolution operators} 
\begin{equation} \textstyle
            [Az](\theta):= \int_{\mathbb{S}^1}W(\theta - \phi)z(\phi)d\phi,
            \label{eq:conv_op}
\end{equation}
where the convolution kernel $W:\mathbb{S}^1\to\R$. 

\begin{definition}[Spatially Invariant Linear System \cite{bamieh2002}, \cite{arbelaiz2024optimalestimationspatiallydistributed}]
Consider a spatiotemporal input-output linear system. 
Let $u(\cdot,t),z(\cdot,t) \in L_2(\mathbb{S}^1)$ be the scalar-valued input and output functions at time $t \in \mathbb{R}_{\geq 0}$, respectively. Let $\theta \in \mathbb{S}^1$ be the spatial coordinate. 
A linear system of the form
    \begin{equation} \textstyle
        \frac{\partial z}{\partial t}(\theta,t) = [Az](\theta,t) + [Bu](\theta,t),
        \label{eq:lin_sys}
    \end{equation}
    is spatially invariant if the linear operators $A$, $B$ are spatially invariant. 
    \label{def:spatial_invt_sys}
\end{definition}

\begin{definition}[Spatial Fourier transform \cite{bamieh2002}, \cite{arbelaiz2024optimalestimationspatiallydistributed}] Let $f,g:\mathbb{S}^1\times\R_{\geq0}$ be spatiotemporal fields with spatial and time coordinates $\theta \in \mathbb{S}^1$ and $t \in \R_{\geq0}$.  Suppose $f(\cdot,t),g(\cdot,t)\in L^2(\mathbb{S}^1)$ for all $t \in \R_{\geq0}$. The \textit{spatial Fourier transform} maps $f(\theta,t)$ into its spatial Fourier coefficients 
\begin{equation}\textstyle
    \hat{f}(k,t) := \int_{\mathbb{S}^1} f(\theta,t) e^{-i 2\pi k\theta}d\theta,
    \label{eq:fourier_coeffs_form}
\end{equation}
where $k\in \Z$ is the spatial frequency. 

 The spatial Fourier transform is a coordinate transformation that expresses $f(\theta,t)$ in terms of the spatial Fourier modes  $\eta_k(\theta) = e^{i2\pi k\theta}$, i.e., the Fourier basis on $\mathbb{S}^1$, and the Fourier coefficients $\hat{f}(k, t)$. The \textit{inverse spatial Fourier transform} can be used to recover $f(\theta,t)$ from its Fourier coefficients $\hat{f}(k,t)$ :
\begin{equation}  \textstyle
    f(\theta, t) = \sum_{k \in \mathbb{Z}} \hat{f}(k, t) e^{i 2\pi k \theta}.  
\end{equation}
Parseval's Identity \cite{parseval_eq} ensures that $\langle \hat{f},\hat{g}\rangle = \langle f, g \rangle$.
The spatial Fourier transform operator is denoted 
by $\mathcal{F}(\cdot)$, and
the inverse spatial Fourier transform operator by 
$\mathcal{F}^{-1}(\cdot)$. 

\end{definition}  

The spatial Fourier transform \eqref{eq:fourier_coeffs_form} \textit{diagonalizes} convolution operators~\cite{bamieh2002}, i.e., if $A$ is a convolution operator~\eqref{eq:conv_op}, then 
$\widehat{[A h]}(k) := \hat{W}(k)\hat{h}(k)$, where $\hat{W}$ is the Fourier transform of the kernel of $A$. Thus, $A$ is mapped by $\mathcal F$ into a multiplication operator over the spatial frequency $k\in \Z$. 
For linear systems of the form \eqref{eq:lin_sys}, if $A$ and $B$ \edit{are convolution operators, i.e. are of the
form (1),} then  
\begin{equation}
    \frac{\partial \hat{z}(k,t)}{\partial t} = \hat{W}_A(k) \hat{z}(k,t) + \hat{W}_B(k)\hat{u}(k,t),
    \label{eq:diagonalization_gen_form}
\end{equation}
where $\hat{W}_A$ and $\hat{W}_B$ are the Fourier transforms of the kernels of $A$ and $B$, respectively. Following~\cite{bamieh2002}, we refer to \eqref{eq:diagonalization_gen_form} as the diagonalization of \eqref{eq:lin_sys}. 

 \begin{definition}[Temporal Laplace Transform] Let $z:\mathbb{S}^1\times\R_{\geq0}$ be a spatiotemporal field with spatial and time coordinates $\theta \in \mathbb{S}^1$ and $t \in \R_{\geq0}$, respectively. Then, the \textit{temporal Laplace transform} maps $z(\theta,t)$ into
 \begin{equation}\textstyle
     [\mathcal{L}z](\theta, s) = \int_0^\infty z(\theta,t)e^{-st}dt,
 \end{equation}
where $s\in \C$, whenever the integral exists. 

 \end{definition}

\section{Opinion Dynamics on the Circle}
\label{sec:model}

We propose a nonlinear opinion dynamics model for an agent making decisions about a continuous distribution of options on the circle. 
For every option $\theta \in \mathbb{S}^1$, $z(\theta, t) \in \R$ is the opinion of the agent for option $\theta$ at time $t$, where the more positive (negative) $z(\theta,t)$ is the more the agent favors (disfavors) option $\theta$. When $z( \theta,t) = 0$, the agent is neutral about option $\theta$. 
%
\edit{Inspired by biological sensory processes \cite{duhamel1997spatial,sridhar2021geometry},} the relationship between each option is encoded by the Lipschitz continuous kernel
$W: \mathbb{S}^1  \to \R$ based solely on their relative positions.
\edit{This design choice is consistent  with other neural field models \cite{bressloff2012spatiotemporal, amari_dynamics_1977, bressloff2002geometric, ermentrout_spatiotemporal_2014, ferreira2016multibump, oubbati2006,dahm1998robot,schoner1995,erlhagen2006,tekulve2024intagent,indiveri2019neuromoprhiag}
and provides analytical tractability.} 
A positive (negative) value of 
$W(\theta - \phi)$ corresponds to excitatory (inhibitory) interactions 
between the 
options $\theta$ and $\phi$. 
The opinion $z(\theta,t)$ evolves according to
\begin{eqnarray}
   \nonumber  \textstyle \tau \frac{{\partial}z}{\partial t}\!(\theta,t)\! \hspace{-0.35em} & \hspace{-0.35em} =  \hspace{-0.35em} & \hspace{-0.35em}
      \textstyle \!  - z(\theta,t) \!+\! \alpha \!\!\int\limits_{\mathbb{S}^1} \! W\!(\theta \!- \!\phi) \hspace{-0.05em} S \hspace{-0.1em}(z(\phi,t)) d\phi  \! + \!u(\theta,t) \\
  \hspace{-0.35em} & \hspace{-0.35em} =  \hspace{-0.35em} & \hspace{-0.35em} \hspace{-0.1em} [Gz](\theta,t) +  u(\theta,t) ,
  \label{eq:CO_DA}
\end{eqnarray}
where $u(\theta,t)\in  \R$ is the input, $\tau \in \R_{>0}$ is the characteristic timescale, and 
$\alpha \in \R_{>0}$ is the attention 
to option interactions, i.e., $\alpha$ models the agent commitment to forming strong opinions. The nonlinear nature of~\eqref{eq:CO_DA} comes from $S: \R \to \R$, a saturating function with $S(0) = 0$, $S^{'}(0) = 1$. 

\section{Spectral Analysis of Linearization}
\label{sec:sys_analysis}

We study the spectrum of the linearization of \eqref{eq:CO_DA} at the neutral equilibrium $z(\theta,t) = 0$, $\forall \theta\in \mathbb{S}^1$. 
\edit{
We prioritize local behavior because it captures key changes in the stability and quantity of equilibria. While global analysis is theoretically valuable, it is often infeasible due to the complexity of the system. Although we do not estimate the region of validity of the local analysis we present in this paper, methods for bounding the region of validity for similar analyses exist, e.g.~\cite{gupta2024estimates}. Conclusions from linearization typically hold within a sufficiently large parametrized neighborhood of the neutral equilibrium at the onset of instability.
The implementations in this paper, which focus on parameter regimes near this critical point, demonstrate the practicality and validity of this local approach.}

We first prove 
spatial invariance,
which enables the linearized system to be \textit{diagonalized}.
Using the diagonalization, we compute the eigenvalues and eigenfunctions of the linearized system and prove their relationship with the Fourier coefficients of the kernel and the spatial Fourier modes.

\begin{lemma}[Spatial invariance of the model linearization]
 Define the nonlinear operator in \eqref{eq:CO_DA} as \edit{$[Fz](\theta,t) = \int_{\mathbb{S}^1}W(\theta - \phi)S(z(\phi,t))$.}
The differential of $F$ in the direction $z$ 
at $z(\theta,t) = 0$ is
\begin{equation}\textstyle
   [A_Fz](\theta,t)\! \hspace{-0.05em}  = \hspace{-0.05em}  \![D_zF(0)](\theta,t) \!=\!  \int_{\mathbb{S}^1} \hspace{-0.05em} W(\theta \!-\! \phi)z(\phi,t)d\phi.
   \label{eq:int_linear}
\end{equation}
The linearization of \eqref{eq:CO_DA} at the neutral equilibrium $z(\theta,t) = 0$,
\begin{equation}
\begin{split}
\textstyle \frac{{\partial}z}{\partial t}(\theta,t)\! 
& = \! \textstyle \frac{1}{\tau}\Big(\! -\!z(\theta,t)+\! \alpha [A_Fz](\theta,t) \! +\! u(\theta,t)\! \Big)\!\\&\textstyle =\! \frac{1}{\tau}([A_Gz](\theta,t) + u(\theta,t)),
\end{split}\label{eq:linearization}    
\end{equation}
is a spatially invariant system in the sense of Definition \ref{def:spatial_invt_sys}.

\label{lemma:sp_inv_sys}
\end{lemma}
\begin{proof} 
Consider the expansion  $S(\epsilon z) = \sum_{n=0}^{\infty}\frac{1}{n!}(\epsilon z)^n S^{(n)}(0) $. Then, we can express
$D_zF(0) = \lim_{\epsilon \to 0}  \frac{\epsilon \int_{\mathbb{S}^1}\! \!W(\theta - \phi)S'(0) z(\phi)d\phi+ \mathcal{O}(\epsilon^2)}{\epsilon}$, 
where $\mathcal{O}(\epsilon^2)$ denotes higher order terms in $\epsilon$.
As $\epsilon \to 0$, the higher order terms vanish and we are left with \eqref{eq:int_linear}. 
Note that $A_F$ is a spatial convolution operator, which is spatially invariant.
Then, by linearity so is $A_G$. Thus, by Definition~\ref{def:spatial_invt_sys}, \eqref{eq:linearization} is a spatially invariant system.
\end{proof}

As a consequence of Lemma \ref{lemma:sp_inv_sys}, we can diagonalize the model linearization \eqref{eq:linearization}. Since \eqref{eq:int_linear} is a convolution operator, we use \eqref{eq:diagonalization_gen_form} to get
\begin{equation} \textstyle
     \frac{{\partial}\hat{z}}{\partial t}(k,t) = \frac{1}{\tau}\big( -1 + \alpha \hat{W}(k) \big) \hat{z}(k,t) + \frac{1}{\tau}\hat{u}(k,t).
     \label{eq:CO_DA_fourier}
\end{equation}

\begin{lemma}[Eigenvalues and eigenfunctions of the linearized system]
The eigenvalues $\lambda_k \in \text{Sp}(A_G)$ of the linearized system \eqref{eq:linearization} can be computed as
\begin{equation}\textstyle
    \lambda_k = \frac{1}{\tau}(-1 + \alpha \hat{W}(k)).
    \label{eq:evals}
\end{equation}
for $k\in \Z$. The corresponding eigenfunctions are the spatial Fourier modes $\eta_k(\theta) = e^{i2\pi k \theta}$.
\label{lemma:evals}
\end{lemma}
\begin{proof} 
The form of the eigenvalues follows directly from the diagonalization \eqref{eq:CO_DA_fourier} of the linearized dynamics \eqref{eq:linearization}. The eigenfunctions are the spatial Fourier modes because they form the basis of the Fourier transformation that is used to diagonalize the system.
\end{proof}

Lemma~\ref{lemma:evals} reaffirms that, because of spatial invariance, the spatial Fourier modes are the eigenfunctions of the model linearization for any kernel design, provided it is spatially-invariant. Since the Fourier coefficients of the kernel determine the eigenvalues associated with each mode, they dictate which modes dominate. More precisely, if all modes are stable, i.e., $\Re(\lambda_k)<0$ for all $k\in\Z$, spatiotemporal inputs $u(\theta,t)$ will be predominantly amplified along the Fourier modes with largest $\Re(\lambda_k)$ as detailed in Section~\ref{sec:input_response}.

When the leading modes become unstable, that is, when the real part of their eigenvalues change from negative to positive through, e.g., an increase of the attention parameter $\alpha$, the nonlinear model~\eqref{eq:CO_DA} undergoes a bifurcation that enables robust opinion formation even in the absence of inputs. 
The leading 
 Fourier modes determine the number of maxima of the stable steady-state opinion patterns emerging at the bifurcation, as detailed in the next section.


\section{
Opinion-Forming Bifurcations} 
\label{sec:bif_theory}

We revisit the results presented in 
\cite{ermentrout_spatiotemporal_2014} for \eqref{eq:CO_DA} with zero input. 
We prove that~\eqref{eq:CO_DA} undergoes a bifurcation and compute the bifurcation point.
A local bifurcation occurs when the number and/or stability of the equilibrium solutions 
changes due to one or more eigenvalues of the model linearization crossing the imaginary axis as a parameter is varied. The state and parameter value at which this occurs is the bifurcation point.
We study how opinion patterns emerge 
and the role of kernel $W$ and show a bistability that enables rapid formation of strong opinions.
We make the following assumption to ensure the eigenvalues of~\eqref{eq:linearization} are real.
\begin{assumption}[Symmetric kernels]
    The kernel $W$ in~\eqref{eq:CO_DA}) is symmetric, i.e. $W(\psi) = W(-\psi)$. In particular, its Fourier coefficients $\hat{W}$ are real and $\hat{W}(k) = \hat{W}(-k)$.
    \label{as:symmetry}
\end{assumption}

\begin{lemma}[Existence of a bifurcation point at the neutral equilibrium] Consider \eqref{eq:CO_DA}. 
Let Assumption \ref{as:symmetry} hold.  \edit{ Suppose $\arg\max_{k\in\Z}\hat{W}(k) = \{-k_{\max}, k_{\max}\}$, }
 \edit{ $k_{\max}\in \N$}, denote the spatial frequency corresponding to the largest $\hat{W}(k)$.
Then, system~\eqref{eq:CO_DA} undergoes a bifurcation at the neutral equilibrium $z(\theta, t) = 0$ and $\alpha^* = \frac{1}{\hat{W}(k_{\max})}$. In particular, for $0< \alpha < \alpha^*$ the neutral equilibrium is locally asymptotically stable and for $\alpha > \alpha^*$ the neutral equilibrium is unstable. The bifurcation branches emerging at bifurcation for $\alpha=\alpha^*$ are tangent to the subspace of $L^2(\mathbb S^1)$ generated by $\cos(2\pi k_{\max}\theta)$ and $\sin(2\pi k_{\max}\theta)$. That is, steady-state opinion patterns along the bifurcation branches have the form $z(\theta,t) = A\cos(2\pi k_{\max}\theta) + B \sin(2\pi k_{\max}\theta)$ for $A,B \in \R$. In particular,  
the number of maxima exhibited in the opinion patterns forming at bifurcation is fixed by $k_{\max}$.

\label{lemma:bif_point}
\end{lemma}
\begin{proof} 
From Lemma \ref{lemma:evals}, $\lambda_k \in \text{Sp}(A_G)$ are given by \eqref{eq:evals}.
We solve for $\lambda_k = \frac{1}{\tau}(-1 + \alpha \hat{W}(k)) = 0$.
Then, the first eigenvalue crossing occurs at  $\alpha^* = \frac{1}{\hat{W}(k_{\max})} = \frac{1}{\hat{W}(-k_{\max})}$ with two eigenvalues crossing at $0$.
For $\alpha< \alpha^*$ we have $\lambda_k < 0$ $\forall k$ so the origin is stable. However, once $\alpha> \alpha^*$, there exist at least two positive eigenvalues so the origin will be unstable. For the linearization, at the bifurcation, $\lambda_k <0$, $\forall k \neq \pm k_{\max}$,  so the corresponding spatial Fourier modes belong to the stable manifold. The spatial Fourier modes $\eta_{ k_{\max}}(\theta)$ and  $\eta_{ -k_{\max}}(\theta)$ form a basis for the center manifold; hence, they determine the dominating bifurcation direction and emerging pattern of the model linearization. 
\end{proof}

The bifurcation of system \eqref{eq:CO_DA} with zero input is  studied in 
\cite{ermentrout_spatiotemporal_2014}. There a perturbation analysis is used to show that the spatial pattern that appears is determined by the leading modes. 
 There are infinitely many branches of non-zero equilibria which exhibit the same pattern with $k_{\max}$ maxima up to spatial translation.
 As discussed in 
 \cite{ermentrout_spatiotemporal_2014}[Section 4.2.1], the stability of the bifurcating branches can be computed as functions of $\hat{W}(k_{\max})$, $S''(0)^2$ and $S'''(0)$. Generally, when $S''(0) = 0$ all of the non-zero branches of equilibria that bifurcate from the origin are stable. 
 When $S''(0) \neq 0$, all non-zero branches bifurcating at the origin are unstable; however, due to higher order terms, stable branches of non-zero equilibria exist farther away from the origin. 
    Fig.~\ref{fig:bif_diag_edited} illustrates $||z||$ at the equilibria as a function of the bifurcation parameter $\alpha$  with 
    \begin{equation}\textstyle
    S(z) = \frac{\tanh(z - \xi) - \tanh(-\xi)}{\sech^2(\xi)},
    \label{eq:shifted_tanh}
\end{equation}
a shifted hyperbolic tangent with $\xi \geq 0$ shift. Note that $S''(0) = 0$ for $\xi = 0$ and $S''(0) \neq 0$ $\forall \xi \neq 0$.

We see that for $\xi \neq 0$, there are regions below the bifurcation point for which the neutral and a non-neutral solution are both stable. 
This \textit{bistable} region enables rapid formation of strong opinions in response to spatially distributed input, as discussed in Section~\ref{sec:input_response}.
The patterns of opinion formation depend on the kernel, which can be designed. Fig.~\ref{fig:ker_four_ss} shows how $k_{\max}$, the spatial frequency corresponding to the largest Fourier coefficient of the kernel, determines the number of maxima exhibited in the steady-state opinion pattern of the agent for \eqref{eq:CO_DA} with zero-input.
Spatial invariance ensures that for all initial conditions the solution converges to the same opinion pattern modulo a spatial translation (Fig.~\ref{fig:ker_four_ss}c).
\edit{

The eigenstructure of the linearization of spatially invariant systems with symmetric kernels (Assumption~\ref{as:symmetry}) is robust to small violations of both spatial invariance and kernel symmetry. Dominant eigenvalues, in particular, remain unique and real.
 }

 \begin{figure} [!t]
    \centering
    \vspace*{5pt}
    {\includegraphics[width=1\linewidth]{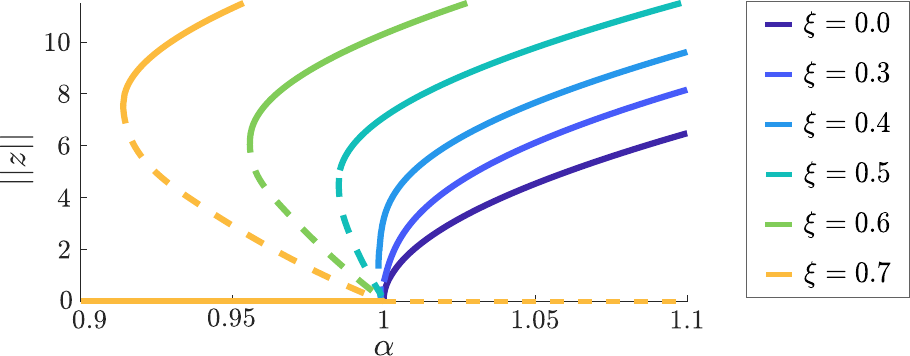}}
    \caption{Bifurcation diagrams illustrating the effect of the shift value $\xi$ on the dynamics \eqref{eq:CO_DA} with shifted sigmoid \eqref{eq:shifted_tanh}. Stable (unstable) branches of equilibria are shown as solid (dashed) lines. 
    }
    \label{fig:bif_diag_edited}
\end{figure}

\section{ 
 \edit{ Decision-Making  on the Circle with Tunable Sensitivity to Distributed Input}}
\label{sec:input_response}

We reintroduce distributed input to the model, and use its linearization, together with spatial and temporal frequency analysis, 
to infer the nonlinear input-output behavior. 

We make the following assumptions.

\begin{assumption}[Shifted sigmoid]
\label{as:sh_sigmoid}
  We assume  $S''(0)\neq 0$ to ensure that a bistable region exists (see Fig.~\ref{fig:bif_diag_edited} for $\xi \neq 0$).
  
\end{assumption}

\begin{assumption}[Input assumptions]
\label{ass:input}
Inputs $u(\theta,t)\in L^2(\mathbb S^1)$ for all $t\geq 0$. Furthermore, for all $\theta\in\mathbb S^1$, $u(\theta,\cdot):\R_{\geq 0}\to\R$ is slowly varying, that is, it is Lipschitz continuous with Lispschitz constant $0<l\ll \tau^{-1}$, for $\tau$ in~\eqref{eq:CO_DA}.
\end{assumption}

The condition $l\ll\tau^{-1}$ implies that inputs vary much more slowly than the characteristic time constant of model~\eqref{eq:CO_DA}. Hence, under Assumption~\ref{ass:input}, we can use the quasi-static input approximation and let $u(\theta,t)\equiv u_h(\theta)$.

For any diagonalized spatially distributed system of the form \eqref{eq:diagonalization_gen_form}, if $[\mathcal{L}\hat{u}](k,s)$, $[\mathcal{L}\hat{z}](k,s)$ exist, then the transfer function $H(k,s)$ characterizes the input-output response in terms of the Laplace transforms of $\hat{z}(k,t)$ and $\hat{u}(k,t)$, i.e., \edit{$[\mathcal{L}\hat{z}](k, s) = H(k,s)[\mathcal{L}\hat{u}](k,s).$}
 By the Final Value Theorem, if the input is constant in time and~\eqref{eq:diagonalization_gen_form} is stable, then $\lim_{t \to \infty} \hat{z}(k,t) \!=\! sH(k,0)\Big(\frac{\hat{u}_h(k)}{s}\Big) \! = \!H(k,0)\hat{u}_h(k)$. This leads us to the following definition.

 \begin{figure} [!t]
    \centering
    \vspace*{5pt}
    {\includegraphics[width=.95\linewidth]{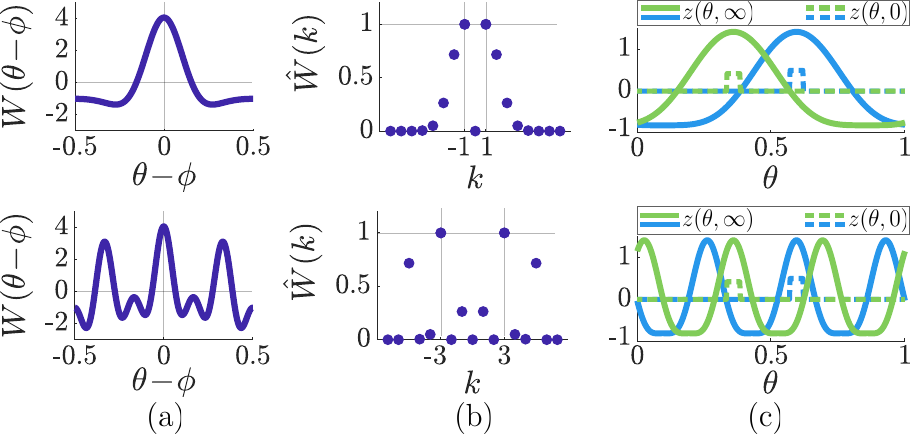}}
    \caption{ Influence of the kernel design on the steady-state opinion patterns of \eqref{eq:CO_DA} with zero-input. (a) Two kernel designs. (b) Fourier coefficients of the kernel. Top: $\pm k_{\max} =\pm1$. Bottom: $\pm k_{\max} = \pm3$. (c) Steady-state opinion pattern $z(\theta,\infty)$, of dynamics \eqref{eq:CO_DA} for initial conditions $z(\theta,0)$. The number of maxima of $z(\theta,\infty)$ equals $k_{\max}$ of the corresponding kernel. Parameters: $\tau = 1$, $\alpha = 0.98$, $p = 3$, $\xi = 0.7$.}
    \label{fig:ker_four_ss}
    
\end{figure}

\begin{definition}[Spatial transfer function] The \textit{spatial transfer function} of~\eqref{eq:diagonalization_gen_form} is
    \begin{equation} \textstyle
        \tilde{H}(k) = H(k, 0).
        \label{eq:sp_tf}
    \end{equation}
\end{definition}
Spatial transfer function $\tilde{H}(k)$ 
determines the steady-state output of~\eqref{eq:diagonalization_gen_form} in response to input that is constant in time.

\begin{theorem} [Spatial transfer function of \eqref{eq:linearization}] Let Assumptions \ref{as:symmetry}--
\ref{ass:input} hold.
\edit{Let $\arg\max_{k\in\Z}\hat{W}(k) = \{-k_{\max}, k_{\max}\}$, }
 \edit{ $k_{\max}\in \N$, denote the spatial frequency 
 of the largest $\hat{W}(k)$}.
Spatial transfer function  \eqref{eq:sp_tf} of the linearized model~\eqref{eq:linearization} is
    \begin{equation} \textstyle
        \tilde{H}(k) = \frac{\tau}{1 -  {\alpha} \hat{W}(k)},\quad k \in \Z.
        \label{eq:sys_sp_tf}
    \end{equation}
    In particular, for $ {\alpha} \nearrow \alpha^*$, $\tilde{H}(\pm k_{\max}) \to \infty$. 
\label{thm:sp_tf}
\end{theorem}

\begin{proof} From Lemma \ref{lemma:evals} we know the eigenvalues of \eqref{eq:CO_DA_fourier}, the diagonalization of~\eqref{eq:linearization}. 
\edit{So, we compute} ${H}(k,s) = \frac{1}{s - \lambda_k}$.
Then $\tilde{H}(k)\!= \! 
     \frac{\tau}{1 -  {\alpha} \hat{W}(k)} .$
 As $ {\alpha} \nearrow \alpha^*$, $H(\pm k_{\max},0)\! \to \! \infty$.
\end{proof}

Theorem \ref{thm:sp_tf} implies that close to bifurcation, i.e., for $ {\alpha} \nearrow \alpha^*$, 
the input-output response of the linearized system is dominated by the leading spatial Fourier modes $\eta_{ \pm k_{\max}}(\theta)$. This means that the \textit{alignment} of $u(\theta)$ with $\eta_{\pm k_{\max}}(\theta)$ is the main determinant of the model response to inputs.

For the nonlinear system with $ {\alpha} \nearrow \alpha^*$, the input-output response takes place in the bistable region. If $\langle \eta_{\pm k_{\max}}(\theta), {u_h(\theta)} \rangle = \hat u_h(k_{\max}) \neq 0$, the \textit{input is aligned with the leading modes} $\eta_{ \pm k_{\max}}(\theta)$. These modes filter input nonlinearity and amplify the input because, by Theorem~\ref{thm:sp_tf}, the direction of these modes are ultrasensitive to input. The result is a steady-state opinion pattern with $k_{\max}$ maxima and  $||z(\theta)|| \gg ||u_h(\theta)||$, as illustrated in the top row of Fig.~\ref{fig:ker_i_o}. As shown in Fig.~\ref{fig:ker_i_o}a (top), the Fourier coefficients of the input for $\pm k_{\max} = \pm 1$ are nonzero meaning the input is aligned with $\eta_{ \pm k_{\max}}(\theta) = \eta_{ \pm1}(\theta)$. 
The input distribution in Fig.~\ref{fig:ker_i_o}b (top) is small in magnitude (less than 0.01), while the resulting steady-state opinion pattern in Fig.~\ref{fig:ker_i_o}c (top) has $k_{\max} = 1$ maximum that is greater than 1.5 in magnitude. 

If  $\langle \eta_{\pm k_{\max} (\theta)}, {u}_h(\theta) \rangle = 0$, \textit{the input is unaligned with the leading modes} $\eta_{ \pm k_{\max}}(\theta)$ and the steady-state opinion pattern will not exhibit large maxima. I.e., because the input does not have a component along the ultrasensitive direction, by Theorem~\ref{thm:sp_tf} it does not get amplified, as illustrated in the bottom row of Fig.~\ref{fig:ker_i_o}. 
As shown in Fig.~\ref{fig:ker_i_o}a (bottom), the Fourier coefficients of the input for $\pm k_{\max} = \pm 1$ are zero meaning the input is unaligned with $\eta_{ \pm k_{\max}}(\theta) = \eta_{ \pm1}(\theta)$. 
The input distribution in Fig.~\ref{fig:ker_i_o}b (bottom) is small in magnitude (less than 0.01) and the resulting steady-state opinion pattern in Fig.~\ref{fig:ker_i_o}c (bottom) has no large maximum.

Our results show how even very small distributed input can trigger the formation of a strong opinion and whether or not this happens depends on the design of kernel $W$. Thus, $W$ can be designed to tune response to be ultrasensitive to inputs that matter for function and robust to inputs that don't.

\begin{figure} 
    \centering
    \vspace*{5pt}
    {\includegraphics[width=.95\linewidth]{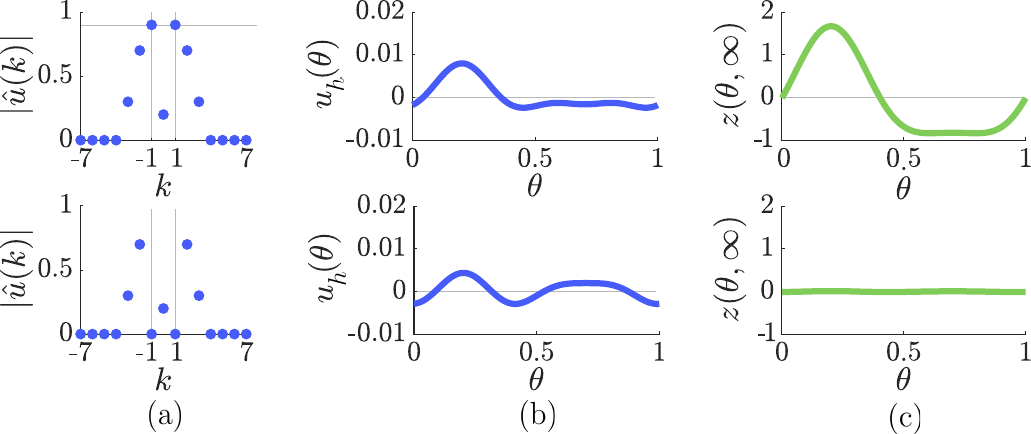}}
    \caption{Input-output behavior of the dynamics \eqref{eq:CO_DA} with input distributions aligned or unaligned with the Fourier mode corresponding to $\pm k_{\max} \!=\! \pm1$. Top row: Aligned. Bottom row: Unaligned. (a) Magnitude of the Fourier coefficients of the input. (b) Input distribution. (c) Steady-state opinion pattern $z(0,\infty)$. Parameters: $\tau \!=\! 1$, $\alpha \!=\! 0.98$, $p \!=\! 3$, $\xi \!=\! 0.7$. }
    \label{fig:ker_i_o}
\end{figure}

\section{Application to Robot Navigation}
\label{sec:application}

We illustrate with \edit{simulations} the benefits of our approach to perceptual decision-making with an application of the dynamics  \eqref{eq:CO_DA} to robot
navigation. We consider the case of a robot 
 moving in a crowded space, such as an airport, where it must pass through gaps of different sizes (e.g., between people in a line) that may change over time. We assume the robot has a (visual) sensor so that it can perceive these gaps.  

We specialize to a scenario where a robot finds itself trapped inside a circle of people and needs to choose and cross through a large enough gap between people.
Choosing a gap is challenging as people may be distributed unevenly around the circle, resulting in multiple gap options, only a select few of which may be suitable for the robot to cross. Also, the size of the gaps may change over time due to people moving for their own purposes or  in response to the robot, e.g., people may move to make space for the robot to cross. 

In Section \ref{sub:kernel_design} we present a framework for designing $W$ from its Fourier coefficients to allow the robot to select a single gap. 
We discuss four scenarios that demonstrate how our model can be used for fast-and-flexible decision-making in this robotic navigation problem. In Section \ref{sub:app_static_line} two scenarios demonstrate the robot's ability to choose a single gap, while in Section \ref{sub:app_dyn_line} the two other scenarios show how the robot can quickly adapt to changes in gap sizes.

We take $\mathbb{S}^1$ to represent the circular visual field for the robot. Then an option $\theta \in \mathbb{S}^1$ represents the angle associated to a point in the visual field.
 The input $u(\theta,t)$ is the visual observation (e.g., pixel) at angle $\theta$ at time $t$.  We let a point in the input distribution that reflects a gap be represented by $u(\theta,t)>0$, in blue in Fig.~\ref{fig:app_one_two} (bottom). We assume changes in gaps occur slowly enough that Assumption~\ref{ass:input} holds. The opinion $z(\theta,t)$ as shown in Fig.~\ref{fig:app_one_two} (top), captures the robot's preference over time for one gap, where the preference corresponds to the strongest opinion (in yellow).

\subsection{Fourier-Based Kernel Design}
\label{sub:kernel_design}

We leverage the results of Theorem~\ref{thm:sp_tf} to design \edit{a} kernel $W$
that imposes the desired 
opinion formation behavior in response to distributed input on $\mathbb{S}^1$.  Options (angles) that are close (far) to each other  should have an excitatory (inhibitory) interaction. And the opinion pattern should have a single maximum, so that the robot selects a single gap. 
From the results summarized in Section~\ref{sec:input_response}, we design $W$ such that \edit{$\hat{W}(k) = 0$ for $k = 0$, and $\hat{W}(k) = \hat{f}(k)$  $\forall k \neq 0$, where $\hat{f}(k):\Z \to \R$ is strictly decreasing, square-summable and symmetric.} The strictly decreasing property ensures that \edit{$\pm k_{\max} = \pm1$} while 
square-summability 
ensures that the inverse Fourier transform of $W$ exists and that $W(\theta - \phi) \in L_2(\mathbb{S}^1)$, by Parseval's identity \cite{parseval_eq}. Symmetry is required to satisfy Assumption \ref{as:symmetry}. 
For the following simulations, we take \edit{a Gaussian function} $\hat{f}(k) = e^{-(k - 1)^2/p^2}$, \edit{where $p$ adjusts its width.}
\begin{figure} 
    \centering
    \vspace*{5pt}
    {\includegraphics[width=1\linewidth]{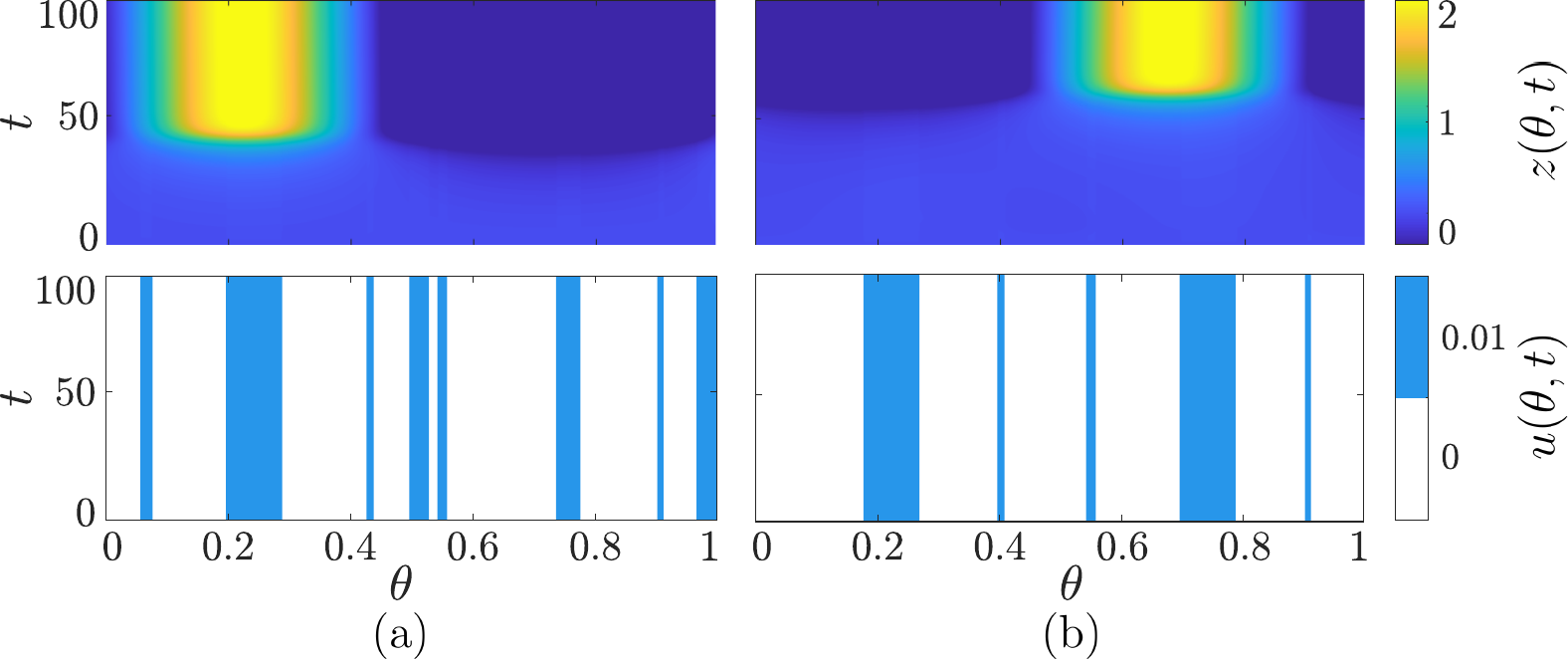}}
    \caption{Decision-making of a robot selecting a gap through which to cross  a circle of non-moving people. Bottom row: gap distribution over time where gaps are indicated by $u(\theta,t)>0$ in blue. Top row: opinion pattern over time (strongest opinion in yellow). (a) One widest gap. (b) Two wide gaps of same size.  \edit{Parameters: $\tau = 1$, $\alpha = 0.98$, $\xi = 0.7$, 
    $p = 3$.}}
    \label{fig:app_one_two}
\end{figure}
\subsection{Choosing the Best Gap and Avoiding Deadlock}
\label{sub:app_static_line}

We illustrate the model's ability to pick the best among multiple gaps and to rapidly avoid deadlock when faced with two equally suitable gaps. We assume  the people in the circle are not moving. 
In Fig.~\ref{fig:app_one_two}a (bottom) there are several gaps but one that is clearly wider than the others. \edit{The input at the location of the widest gap gets amplified so that the single maximum guaranteed by the kernel design discussed in Section \ref{sub:kernel_design} forms at that location. Hence, we see} in Fig.~\ref{fig:app_one_two}a (top), \edit{that} the robot forms a strong preference for the widest gap. 
In Fig.~\ref{fig:app_one_two}b (bottom), there are two equally wide gaps. \edit{Since the kernel design discussed in Section \ref{sub:kernel_design} ensures that only one maximum is formed, 
one of the inputs gets amplified and the others suppressed. }
In Fig.~\ref{fig:app_one_two}b (top), the robot forms strong opinions for one of the two widest gaps and avoids deadlock.

\subsection{Robustness and Responsiveness to Change}
\label{sub:app_dyn_line}

We illustrate the model's robustness to unimportant change and responsiveness to important change in input.  We assume the people in the circle are moving. In Fig.~\ref{fig:app_move_resist} (bottom), there is initially one very wide gap and one narrow gap. However, over time, the wide gap becomes narrower, and the narrower gap becomes wider. In Fig.~\ref{fig:app_move_resist}a, the decrease in size of the initially wide gap is small enough that the robot can still fit through it and thus it does not change its choice.
\edit{Such a change in gap size could result from humans making only small positioning adjustments in response to the robot, which would reflect as small perturbation to the input distribution.}
\edit{In this case, since a strong opinion first forms in favor of the gap that is initially widest and the gap remains sufficiently large, the robot does not change its mind. This illustrates the robustness of the decision-making to small changes in input.} 

In Fig.~\ref{fig:app_move_resist}b, the decrease in size of the initially wide gap is large enough that the robot changes its choice to the other emerging gap. 
\edit{Such a change in gap size could result from humans trying to make space for the robot to pass, which would reflect as large change to the input distribution.}
\edit{ In this case, 
the opinion pattern forms changes in favor of the emerging widest gap, i.e, the robot changes its mind about which gap it prefers. This illustrates the adaptability of the robot's decision-making to large changes in input. 
We plan to characterize the threshold that governs the switch between the behaviors shown in Fig.~\ref{fig:app_move_resist}a and Fig.~\ref{fig:app_move_resist}b in future work. }



\begin{figure} 
    \centering
    \vspace*{5pt}
    {\includegraphics[width=1\linewidth]{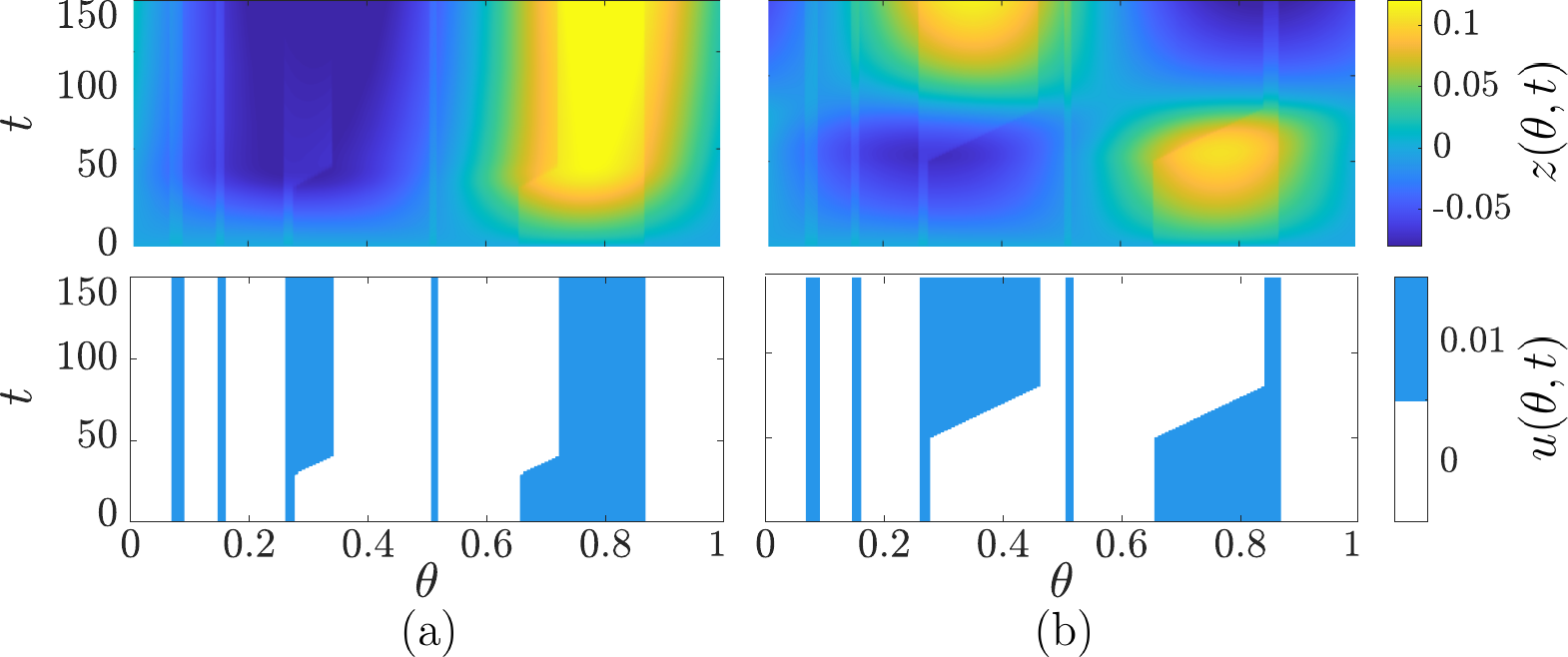}}
    \caption{Decision-making of a robot selecting a gap through which to cross  a circle of moving people.  Bottom row: gap distribution over time where gaps are indicated by $u(\theta,t)>0$ in blue. Top row: opinion about where to cross the line over time (strongest opinion in yellow). (a) Small decrease  over time in size of initially widest gap. (b) Large decrease  over time in size of initially widest gap. \edit{Parameters: $\tau \!=\! 1$, $\alpha \!=\! 0.96$, $\xi \!= \!0.6$, 
    $p \!= \!3$.}}
    \label{fig:app_move_resist}
\end{figure}
\section{Discussion}
\label{sec:discussion}

We presented a \edit{new} nonlinear opinion dynamics model for an agent making decisions about a continuous distribution of options in response to distributed input on the circle. We proved spatial invariance of the model linearization and a bifurcation of the model with zero input, which yields fast and flexible decision-making. \edit{A key contribution is our study of} the input-output behavior of the model and design of the kernel. We demonstrated \edit{important} advantages of the model in robot perceptual decision-making problem. In future work we aim to \edit{derive an estimate for the region of validity of the model linearization}
 \edit{and characterize the relationship between input distribution and the location where the maximum in the opinion distribution form. We will} implement this model for perceptual decision-making in robotics where the dynamics are in a closed loop with the physical dynamics of the agent. 

\section*{Acknowledgements}

We thank 
Dr. Juncal Arbelaiz (Princeton University) for insights and references on spatially-invariant systems.




\bibliographystyle{./bibliography/IEEEtran}
\bibliography{./bibliography/references}

\end{document}